\documentclass[12pt]{amsart}
\pdfoutput=1

\title{Normal trees of digraphs}

\author{Florian Reich}
\address{Universit\"at Hamburg, Department of Mathematics, Bundesstrasse 55 (Geomatikum), 20146 Hamburg, Germany}
\email{florian.reich@uni-hamburg.de}

\keywords{normal tree, infinite graph, infinite digraph}

\usepackage{amsmath,amssymb,amsthm}
\usepackage{mathtools}
\usepackage{tikz}
\usepackage{comment}
\usepackage{enumitem}
\usepackage{thm-restate}

\usepackage{xcolor}
\usepackage[unicode]{hyperref}
\hypersetup{
	colorlinks,
    linkcolor={red!60!black},
    citecolor={green!60!black},
    urlcolor={blue!60!black},
}
\usepackage[abbrev, msc-links]{amsrefs}
\usepackage[capitalize, nameinlink]{cleveref}

\usepackage[utf8]{inputenc}
\usepackage[T1]{fontenc}
\usepackage{lmodern}
\usepackage[babel]{microtype}
\usepackage[english]{babel}

\usepackage{geometry}
\usepackage[presets={vec-cev}]{letterswitharrows}

\usepackage{enumitem}

\let\polishlcross=\l
\def\l{\ifmmode\ell\else\polishlcross\fi}

\let\emptyset=\varnothing

\let\theta=\vartheta
\let\rho=\varrho
\let\phi=\varphi

\def\NN{\mathbb N}

\def\RR{\mathbb R}
 
\newcommand{\Set}[1]{{\left\lbrace {#1} \right\rbrace}}

\def\set#1:#2{\Set{{#1} \colon {#2}}}

\newcommand{\Up}[1]{\lfloor #1 \rfloor}
\newcommand{\OUp}[1]{\mathring{\lfloor #1 \rfloor}}
\newcommand{\Down}[1]{\lceil #1 \rceil}
\newcommand{\ODown}[1]{\mathring{\lceil #1 \rceil}}

\newcommand{\dmtop}{\textsc{DTop}}
\newcommand{\mmtop}{\textsc{MTop}}

\DeclareMathOperator{\vstart}{start}
\DeclareMathOperator{\vend}{end}

\theoremstyle{plain}
\newtheorem{thm}{Theorem}[section]

\newtheorem{prop}[thm]{Proposition}
\newtheorem{cor}[thm]{Corollary}
\newtheorem{lemma}[thm]{Lemma}

\theoremstyle{definition}

\begin{document}
	
\begin{abstract}
	In this paper, we investigate normal trees of directed graphs, which extend the fundamental concept of normal trees of undirected graphs.
	
	We prove that a directed graph $D$ has a normal spanning tree if and only if the topological space $|D|$ is metrizable, which generalises Diestel's result for undirected graphs.
	Furthermore, we show that the existence of normal arborescences implies the existence of normal trees in directed graphs, and that the converse is generally not true.
\end{abstract}
	
\maketitle

\section{Introduction}
Normal trees of undirected graphs are one of the most important tools in infinite graph theory.
A rooted tree $T$ in an undirected graph $G$ is called \emph{normal} if for every connected subgraph $H$ of~$G$ and every two $\leq_T$-incomparable vertices $u,v \in V(H) \cap V(T)$ there exists $w \in V(H) \cap V(T)$ with $w \leq_T u,v$~\cite{DiestelBook2016}.
By definition, normal trees capture how undirected graphs can be separated by finite vertex sets
and thus, in particular, display the ends of their host graphs.

Bowler and the author recently introduced a notion of normal trees in the general setting of connectoids~\cite{connectoids1}, which in particular defines normal trees for directed graphs.
A \emph{weak normal tree~$T$} of a directed graph $D$ is a rooted, undirected tree $T$ with $V(T) \subseteq V(D)$ such that
\begin{itemize}
	\item for every strongly connected subgraph $H$ of $D$ and every two $\leq_T$-incomparable vertices $u, v \in V(H) \cap V(T)$ there exists $w \in V(H) \cap V(T)$ such that $w \leq_T u, v$, and
	\item for every two $\leq_T$-comparable elements $u \leq_T v$ there exists a strongly connected subgraph of $D$ containing $u$ and $v$ that avoids every element $w <_T u$,
\end{itemize}
where $\leq_T$ refers to the tree-order of $T$~\cite{connectoids1}.

Although weak normal trees have the same separation properties as normal trees of undirected graphs, they generally do not display the ends of their host graphs, as defined by B\"urger and Melcher \cite{burger2020ends}:
A strongly connected directed graph $N$ is called a \emph{necklace} if there exists a family~$(H_n)_{n \in \NN}$ of finite strongly connected subgraphs such that $N = \bigcup_{n \in \NN} H_n$ and $H_i \cap H_j \neq \emptyset$ holds if and only if $|i - j| \leq 1$ for every $i, j \in \NN$.\footnote{B\"urger and Melcher introduced necklaces in \cites{burger2020ends} using an equivalent definition.}
The \emph{ends} of a directed graph $D$ are the equivalence classes of necklaces in $D$, where two necklaces $N_1$ and $N_2$ are equivalent if there are infinitely many disjoint $N_1$--$N_2$~paths and infinitely many disjoint $N_2$--$N_1$~paths.

To ensure that normal trees display the ends of their host graphs we further require:
A weak normal tree $T$ of $D$ is a \emph{normal tree} if for every rooted ray $R$ in $T$ there exists a necklace $\eta(R)$ in~$D$ that contains \emph{almost all}, i.e.\ all but finitely many, elements of $V(R)$~\cite{connectoids1}.
\begin{thm}[\cite{connectoids1}*{Theorem~5.2}]\label{thm:endfaithfullness}	
	Let $D$ be a directed graph and let $T$ be a normal tree of~$D$.
	Then $\eta$ induces a bijection between the ends of $T$ and the ends of $D$ that are in the closure of~$V(T)$.
\end{thm}
\noindent
An end $\omega$ is in the \emph{closure} of some set $W$, if for every finite set $X \subseteq V(D)$ there exists a necklace in $\omega$ that avoids $X$ and intersects $W$.

The following theorem shows that normal trees can be characterised in the same way as their undirected counterparts.
Therefore, despite their specific definition, they are a natural extension of normal trees of undirected graphs.
\begin{thm}[\cite{connectoids2}*{Theorems 1.1, 1.2, 3.2}] \label{thm:known_characterisation_jung}
	Let $D$ be a strongly connected directed graph and let $U \subseteq V(D)$. Then the following are equivalent:
	\begin{enumerate}[label=(\arabic*)]
		\item\label{itm:weak_normal_tree} there is a normal tree containing $U$,
		\item\label{itm:normal_tree} there is a weak normal tree containing $U$,
		\item\label{itm:countable_union} $U$ is a countable union of dispersed sets,
		\item\label{itm:countable_union_tk_dispersed} $U$ is a countable union of fat topological split $\vec{K}_{\aleph_0}$-dispersed sets, and
		\item\label{itm:ends} for every end $\omega \in \Omega(D)$ there is a finite set $X_\omega \subseteq V(D)$ such that the strong component $C(X_\omega, \omega)$ has a normal tree containing $U \cap C(X_\omega, \omega)$.
	\end{enumerate}
\end{thm}
\noindent
Here, the terms \emph{dispersed set}, \emph{fat topological split $\vec{K}_{\aleph_0}$-dispersed set} and $C(X_\omega, \omega)$ are directed analogues of their undirected counterparts.

In this paper, we continue to investigate the existence of normal trees of directed graphs.
First, we compare normal tree with the notion of \emph{normal arborescences}~\cite{burger2020ends3}, which was introduced by B\"urger and Melcher and defines normal structures within a directed graph:
An \emph{arborescence}  is a rooted directed graph whose underlying undirected graph is a tree such that all edges are oriented away from the root.
The \emph{normal assistant} of $A$ in $D$ is the auxiliary directed graph obtained from $A$ by adding an edge $uv$ for every two $\leq_A$-incomparable elements $u,v \in V(A)$ for which there exists a path from $\Up{u}_A$ to $\Up{v}_A$ in $D$ that is internally disjoint to $V(A)$~\cite{burger2020ends3}.
The arborescence $A$ is \emph{normal} if its normal assistant is acyclic~\cite{burger2020ends3}.
We prove:
\begin{restatable}{lemma}{NSAimpliesNST}\label{lem:nsa_implies_nst}
	Let $D$ be a directed graph and let $U \subseteq V(D)$.
	If there is a (rayless) normal arborescence in $D$ containing $U$ then there exists a (rayless) normal tree of $D$ containing $U$.
\end{restatable}
\noindent
Moreover, we show that the converse is generally not true.

Second, we characterise the existence of normal \emph{spanning} trees, i.e.\ normal trees $T$ with $V(T)=V(D)$, via the space $|D|$ with B\"urger and Melcher's topology \dmtop ~(see~\cite{burger2020ends2} for an introduction).
Diestel~\cite{diestel2006end}*{Theorem~3.1} showed that there exists a normal spanning tree in an undirected graph $G$ if and only if the space $|G|$ with topology  \mmtop~is metrizable.
We extend Diestel's result to directed graphs:
\begin{restatable}{thm}{Metrizable}\label{thm:metrizable}
	Let $D$ be a directed graph.
	In \dmtop, $|D|$ is metrizable if and only if $D$ has a normal spanning tree.
\end{restatable}

This paper is organised as follows:
We present basic properties of normal trees in~\cref{sec:properties} and compare normal trees and normal arborescences in~\cref{sec:normal_arborescences}.
Finally, we prove \cref{thm:metrizable} in \cref{sec:metrizable}.

\section{Basic properties}\label{sec:properties}
Normal trees of directed graphs have the same separation property as normal trees of undirected graphs by the first condition of (weak) normal trees.
We begin by stating that the deletion of certain vertex sets (i.e.\ $\leq_T$-down-closed sets) of a normal tree splits a directed graph into strong components in the same way as for undirected graphs.
Given a weak normal tree $T$ and some $t \in V(T)$, let $\Down{t}_T:= \{s \in V(T): s \leq_T t \}$, $\ODown{t}_T:= \Down{t}_T \setminus \{t\}$ and similarly $\Up{t}_T, \OUp{t}_T$.
Further, let $C_t^T$ be the unique strong component of $D - \ODown{t}_T$ that contains $t$.

\begin{prop}[\cite{connectoids1}*{Proposition 2.4}]\label{prop:equivalence_weak_normal_tree}
	Let $D$ be a directed graph and $T$ a rooted, undirected tree with $V(T) \subseteq V(D)$. Then the following properties are equivalent:
	\begin{enumerate}[label=(\alph*)]
		\item\label{itm:equivalence_weak_normal_tree_1} $T$ is a weak normal tree of $D$, and
		\item\label{itm:equivalence_weak_normal_tree_2} $C_t^T \cap V(T) = \Up{t}_T$ holds for every vertex $t \in T$.
	\end{enumerate}
	Furthermore, $C_t^T = \Up{t}_T$ holds for every vertex $t$ of a weak normal spanning tree $T$.
\end{prop}

Given a finite set $X$ of vertices, we set $C(X, \omega)$ to be the unique strong component of $D - X$ that contains necklaces of $\omega$.
Given a necklace $N$ and a finite set $X \subseteq V(D)$ there exists a strong component of $N - X$ that contains almost all vertices of $N$, which in turn is again a necklace \cite{connectoids1}*{Proposition~2.3}.
We call this unique strong component the \emph{$X$-tail} of $N$~\cite{connectoids1}.
We say a set $U$ of vertices is \emph{dispersed} if every necklace has finite intersection with $U$~\cite{connectoids2}.
\begin{prop}[\cite{connectoids2}*{Corollary 3.1}] \label{cor:known_dispersed_jung}
	Let $D$ be a strongly connected directed graph and $U \subseteq V(D)$.
	Then $U$ is dispersed if and only if there exists a rayless normal tree of $D$ that contains $U$.
\end{prop}

We introduce a notion of minors in directed graphs that includes notions of minors like butterfly minors and strong minors.
Let $H$ and $D$ be directed graphs. A \emph{broad minor model} of $H$ in $D$ is a family of disjoint subsets $(X_v)_{ v \in V(H)}$ of $V(D)$ and a family $(x_v)_{v \in V(H)}$ of vertices in $V(D)$ with $x_v \in X_v$ for every $v \in V(H)$ such that for each edge $uv \in E(H)$ there exists an $x_u$--$x_v$~path in $D[X_u \cup X_v]$. We say $H$ is a \emph{broad minor} of $D$ if there exists a broad minor model of $H$ in $D$.

\begin{prop}\label{prop:broad_minor}
	Let $D$ be a directed graph and $H$ a broad minor of $D$ witnessed by the minor model $(X_v)_{ v \in V(H)}$, $(x_v)_{v \in v(H)}$. If a set $U \subseteq V(D)$ is dispersed, then the set $W:= \{v \in V(H): x_v \in U \}$ is dispersed in $H$.
\end{prop}
\begin{proof}
	Suppose for a contradiction that there exists a necklace $N$ in $H$ that intersects $W$ infinitely. Let $(H_n)_{n \in \NN}$ be a witness of $N$ and let $n \in \NN$ be arbitrary.
	We consider for each $uv \in E(H_n)$ an $x_u$--$x_v$~path $P_{uv}$ in $D[X_u \cup X_v]$ and let $H_n'$ be the strongly connected subgraph of $D$ induced by the paths $P_{uv}$ for $uv \in E(H_n)$.
	Note that $\{x_v: v \in V(H_n)\} \subseteq V(H_n')$ and that $V(H_n') \subseteq \bigcup_{v \in V(H_n)}  X_v $, by construction.
	
	Thus $(H_n')_{n \in \NN}$ is a witness of a necklace $N'$ in $D$.
	In particular, $N'$ contains every vertex of $\{x_v: v \in W \cap N\}$.
	As $N$ intersects $W$ infinitely often, $N'$ intersects $U$ infinitely often, a contradiction to the dispersedness of $U$.
\end{proof}

\noindent
From \cref{thm:known_characterisation_jung,prop:broad_minor} we can deduce:
\begin{cor}
	The existence of normal spanning trees is closed under taking strongly connected broad minors, and thus under taking strongly connected subgraphs, butterfly minors and strong minors.
\end{cor}

\section{Relation to normal arborescences}\label{sec:normal_arborescences}
In this section we compare the notion of normal trees with the notion of normal aborescences, which was introduced by B\"urger and Melcher in \cite{burger2020ends3}.
We show that, given some set of vertices~$U$, the existence of a normal arborescence containing $U$ implies the existence of a normal tree containing~$U$.
Further, we show that the converse is generally not true.

Normal trees and normal arborescences share the same separation property:

\begin{prop}[\cite{burger2020ends3}*{Lemma~3.4}]\label{prop:normal_arborescence_separation}
	Let $D$ be a directed graph and let $A$ be a normal arborescence in $D$.
	For every strongly connected subgraph $C$ of $D$ and every two $\leq_A$-incomparable elements $u, v \in V(C) \cap V(A)$ there exists $w \in V(C)$ such that $w \leq_A u, v$, where $\leq_A$ refers to the tree order of $A$.
\end{prop}

As a normal arborescence is a subgraph of its host graph, it has more structural information about connectivity in terms of directed paths from the root to its vertices.
Nevertheless, this comes with a less precise description of strong components:
Given a normal arborescence $A$ in $D$ and some $a \in V(A)$, every strong component of $D - \ODown{a}_A$ intersecting $\Up{a}_A$ has to be contained in $\Up{a}_A$ by \cref{prop:normal_arborescence_separation}, but $D[\Up{a}_A]$ does not have to be a strong component itself.
In contrast to this, given a normal spanning tree $T$ of $D$, $D[\Up{t}_T]$ is a strong component of $D - \ODown{t}_T$ for every $t \in V(T)$ by \cref{prop:equivalence_weak_normal_tree}.

\NSAimpliesNST*

\begin{proof}
	Let $A$ be a normal arborescence in $D$ containing $U$.
	By \cref{thm:known_characterisation_jung}, it suffices to show that the distance classes $(X_n)_{n \in \NN}$ of $A$ are dispersed.
	More precisely, we show that every necklace~$N$ has finite intersection with $X_n$ for every $n \in \NN$.
	
	If there exists $a \in V(A)$ such that the $\Down{a}_A$-tail of $N$ is disjoint to $V(A)$, then $N$ has finite intersection with $V(A)$ and in particular, finite intersection with $X_n$ for every $n \in \NN$.
	Therefore we can assume that the $\Down{a}_A$-tail of $N$ contains a vertex of $V(A)$ for every $a \in V(A)$.
	
	We construct a rooted ray $(a_n)_{n \in \NN}$ in $A$ such that the $\ODown{a_n}_A$-tail of $N$ intersects $V(A)$ only in $\Up{a_n}_A$.
	Then the set $X_n$ has finite intersection with $N$ for every $n \in \NN$ as the $\ODown{a_{n + 1}}$-tail of $N$ intersects $V(A)$ only in $\Up{a_{n+1}}_A$, which is disjoint to $X_n$ since $(a_n)_{n \in \NN}$ is strictly $\leq_A$-increasing.
	
	Let $a_1$ be the root of $A$ and suppose that $(a_n)_{n \leq m}$ has been constructed for some $m \in \NN$.
	The $\Down{a_m}_A$-tail contains a vertex of $A$, by assumption.
	Further, since the $\ODown{a_m}_A$-tail intersects $V(A)$ only in $\Up{a_m}_A$, the $\Down{a_m}_A$-tail of $N$ intersects $V(A)$ only in $\OUp{a_m}_A$.
	By \cref{prop:normal_arborescence_separation}, there exists a child $a_{m + 1}$ of $a_m$ such that the $\Down{a_m}_A$-tail of $N$ intersects $V(A)$ only in $\Up{a_{m+1}}_A$.
	Then $a_{m+1}$ is as desired since $\ODown{a_{m+1}}_A = \Down{a_m}$.
	This finishes the construction of $(a_n)_{n \in \NN}$.
	
	If $A$ is rayless, then the construction of the sequence $(a_n)_{n \in \NN}$ terminates, i.e.\ every necklace has finite intersection with $V(A)$, and by \cref{cor:known_dispersed_jung} there exists a rayless normal tree containing~$U$.
\end{proof}
\begin{prop}\label{prop:normal_arborescence_counterexample}
	There exists a directed graph with a normal spanning tree but without a normal spanning arborescence.
\end{prop}
\begin{proof}	
Let $D$ be the directed graph with
$ V(D):=\{a_{\alpha}: \alpha < \omega_1 \} \cup \{b_{\alpha}: \alpha < \omega_1 \}$ and
$$ E(D):= \{a_{\alpha} a_{\beta}: \alpha < \beta < \omega_1 \} \cup \{b_{\beta} b_{\alpha}: \alpha < \beta < \omega_1 \} \cup \{a_{\alpha} b_{\alpha}: \alpha < \omega_1  \} \cup \{b_0a_0\}.$$

First, we show that $D - a_0$ is acyclic:
Consider the unique linear order $\leq$ of $V(D)$ satisfying $a_\alpha < a_\beta$ and $b_\beta < b_\alpha$ for every $\alpha < \beta < \omega_1$ and satisfying $a_\alpha < b_{\alpha'}$ for every $\alpha, \alpha' < \omega_1$.
Any edge of $E(D - a_0)$ has the property that its tail precedes its head in $\leq$.
Thus $D - a_0$ is acyclic and in particular, every strong component of $D - a_0$ is a singleton.
Thus the star with root $a_0$ and leaves in $V(D - a_0)$ is a normal spanning tree of $D$.

We suppose for a contradiction that $D$ contains a normal spanning arborescence $A$. Let $\gamma < \omega_1$ be such that either $a_\gamma$ or $b_\gamma$ is the root of $A$.
The arborescence $A$ contains uncountably many edges in $\{ a_{\alpha} b_{\alpha}: \alpha < \omega_1 \}$:
Otherwise there exists $\gamma < \beta< \omega_1$ such that $A$ contains no edge in $\{ a_{\alpha} b_{\alpha}: \beta \leq \alpha < \omega_1 \}$.
But every path from $a_\gamma$ or $b_\gamma$ to $b_\beta$ has to contain an edge in $\{ a_{\alpha} b_{\alpha}: \beta \leq \alpha < \omega_1 \}$ and thus $A$ is not spanning.

We consider the uncountable set $\{a_\alpha: a_\alpha b_\alpha \in E(A)\}$.
Since the branches of $A$ are countable, there exist $\alpha' < \alpha'' < \omega_1$ with $a_{\alpha'}, a_{\alpha''} \in \{a_\alpha: a_\alpha b_\alpha \in E(A)\}$ such that $a_{\alpha'}, a_{\alpha''} $ are $\leq_A$-incomparable.
The edges $a_{\alpha'} a_{\alpha''}$ and $b_{\alpha''} b_{\alpha'}$ connect $\Up{a_{\alpha'}}_A$ and $\Up{a_{\alpha''}}_A$ in both directions, contradicting the normality of $A$.
This finishes the proof.
\end{proof}
We remark that the proof of \cref{prop:normal_arborescence_counterexample} does not rely on the fact that $D$ does not have ends, i.e.\ given some cardinal $\kappa$ there exists a directed graph satisfying \cref{prop:normal_arborescence_counterexample} that has $\kappa$ many ends: attach a disjoint family of necklaces of cardinality $\kappa$ to the vertex $a_0$ in $D$.

We can deduce from the proof of \cref{prop:normal_arborescence_counterexample} that, in contrast to normal trees in undirected and directed graphs, the existence of a normal arborescence is not closed under taking strongly connected subgraphs:
\begin{prop}
	The existence of normal spanning arborescences is not closed under taking strongly connected subgraphs.
\end{prop}
\begin{proof}
	We consider the directed graph $D$ in the proof of \cref{prop:normal_arborescence_counterexample}.
	Let $D'$ be the directed graph obtained from $D$ by adding the edge set $F:= \{a_0b_\alpha: \alpha < \omega_1 \}$.
	We show that the spanning arborescence $A$ with root $a_0$ and edge set $F \cup \{a_0a_\alpha: 1 \leq \alpha < \omega_1\}$ is normal.
	
	As all vertices except the root have distance one to the root in $A$,
	the normal assistant consists precisely of all edges of $A$ and all edges of $D' - a_0 = D - a_0$.
	Since $D - a_0$ is acyclic and as $a_0$ has in-degree zero in the normal assistant, the normal assistant is acyclic.
	Thus $A$ is a normal arborescence of $D'$, but the strongly connected subgraph $D$ of $D'$ does not have a normal arborescence. 
\end{proof}

\section{Characterisation via metrizable space $|D|$}\label{sec:metrizable}
Diestel proved \cite{diestel2006end}*{Theorem~3.1} that the topological space $|G|$ of an undirected graph $G$ is metrizable in the common topology \mmtop~if and only if $G$ has a normal spanning tree.
For directed graphs B\"urger and Melcher established the space $|D|$ together with the topology \dmtop; see~\cite{burger2020ends2} for an introduction.
In this section we show that Diestel's characterisation carries over to directed graphs:
\Metrizable*

The proof of~\cref{thm:metrizable} is structured as follows.
We begin by proving the forward implication in \cref{lem:metrizable_implies_nst}.
For the backward implication, given a normal spanning tree $T$ of a directed graph~$D$, we define a function $d: |D| \times |D| \mapsto \RR_{\geq 0}$ based on $T$ and prove that $d$ is a metric of $|D|$.
Finally, we verify that $d$ induces the topology \dmtop~of $|D|$.
Thus the existence of a normal spanning tree implies that $|D|$ is metrizable in \dmtop.

\begin{lemma}\label{lem:metrizable_implies_nst}
	If $|D|$ is metrizable in \dmtop, then $D$ has a normal spanning tree.
\end{lemma}
\begin{proof}
	Let $d$ be a metric of $|D|$ that induces \dmtop.
	We define $V_n:= \{v \in V(D): d(v, \omega)> \tfrac{1}{n} \; \; \forall \omega \in \Omega(D) \}$ for all $n \in \NN$.
	Then $V(D)=\bigcup_{n \in \NN} V_n$ since each vertex has an $\epsilon$-ball in \dmtop~that contains no ends.
	We show that each set $V_n$ is dispersed.
	Then $D$ has a normal spanning tree by \cref{thm:known_characterisation_jung}.
	
	Let $n \in \NN$ be arbitrary.
	Further, let $N$ be an arbitrary necklace in $D$ and let $\omega$ be the end that contains $N$.
	Since $d$ induces \dmtop, there exists a finite set $X \subseteq V(D)$ such that $C(X, \omega)$ is contained in $B_{\tfrac{1}{n}}(\omega)$.
	Thus the $X$-tail of $N$ is contained in $C(X, \omega)$ and, in particular, the $X$-tail of $N$ avoids $V_n$.
	Thus $N$ contains only finitely many vertices of $V_n$, which implies that $V_n$ is dispersed.
\end{proof}

We turn our attention to the backward implication.
Let $D$ be an arbitrary directed graph with a normal spanning tree $T$.
Before defining the desired metric $d: |D| \times |D| \mapsto \RR_{\geq 0}$, we define $\vstart(p)$, $\vend(p) \subseteq V(D) \cup \Omega(D)$ for each point $p$ in $|D|$ and a length $\ell$ of each edge of $T$.

For a vertex $p$ or an end $p$, we set $\vstart(p) = \vend(p) := p$.
Given an inner point $p = (uv, \lambda)$ of a (limit) edge $uv$, we set $\vstart(p):= u$ and $\vend(p):= v$.
Further, let $\lambda(p):= \lambda \in (0,1)$.
The function $\ell: E(T) \rightarrow (0,\frac{1}{2}]$ is defined in the following way:
Let $n \in \NN$ be arbitrary and let $E_n \subseteq E(T)$ be the set of edges whose endvertices have distance $n -1$ and $n$ to the root $r$ of $T$ in $T$.
We set $\ell(e):= \frac{1}{2^n}$ for every $e \in E_n$.

Given some edge $e \in E(T)$, we describe a canonical partition of $V(D) \cup \Omega(D)$ into two classes induced by $T - e$:
The graph $T - e$ consists of two components $T_1$ and $T_2$.
Note that each end of $D$ corresponds to an end of $T$ by \cref{thm:endfaithfullness}.
Let $V(T_1) \cup \Omega(T_1)$ and $V(T_2) \cup \Omega(T_2)$ be the desired partition classes, which we refer to as the \emph{sides} of $T - e$.

Given two points $p$ and $q$ of $|D|$, we define a \emph{weight} $w_{\{p,q\}}(e) \in [0,1]$ of $e$:
\begin{enumerate}[label=(\roman*)]
	\item\label{itm:weight_1} \textbf{One side of $T - e$ contains no element of $\{\vstart(p), \vstart(q), \vend(p), \vend(q)\}$:}
		We set $w_{\{p,q\}}(e):= 0$.
	\item\label{itm:weight_2}  \textbf{One side of $T - e$ contains exactly one element $x$ of $\{\vstart(p), \vstart(q), \vend(p), \vend(q)\}$:}
		We set $w_{\{p,q\}}(e):= \begin{cases}
				1 - \lambda(p) & \text{if } x = \vstart(p),\\
				\lambda(p)		& \text{if } x = \vend(p),\\
				1 - \lambda(q) & \text{if } x = \vstart(q),\\
				\lambda(q)		& \text{if } x = \vend(q).
				\end{cases}
		$
	\item\label{itm:weight_3}  \textbf{One side of $T - e$ contains $\vstart(p), \vend(p)$ and the other side of $T - e$ contains $\vstart(q), \vend(q)$:}
		We set $w_{\{p,q\}}(e):= 1$.
	\item\label{itm:weight_4}  \textbf{One side of $T - e$ contains $\vstart(p), \vstart(q)$ and the other side of $T - e$ contains $\vend(p), \vend(q)$:}
		We set $w_{\{p,q\}}(e):= |\lambda(p) - \lambda(q)|$.
	\item\label{itm:weight_5}  \textbf{One side of $T - e$ contains $\vstart(p), \vend(q)$ and the other side of $T - e$ contains $\vend(p), \vstart(q)$:}
		We set $w_{\{p,q\}}(e):= 1 -  |\lambda(p) - \lambda(q)|$.
\end{enumerate}
\noindent
We show that the definition of $w$ indeed only relies on $\lambda$ being defined for internal points of edges:
Let $p$ be either a vertex or an end of $D$ and let $q$ be an arbitrary point of $|D|$.
Then only cases~\labelcref{itm:weight_1,itm:weight_2,itm:weight_3} apply to $p$ and $q$.
Furthermore, in case~\labelcref{itm:weight_2} $x$ can only be $\vstart(q)$ or $\vend(q)$.
Since $w$ is symmetric, we can conclude that it suffices to define $\lambda$ only for internal points of edges.

We are now ready to define
$$d(p,q):= \sum_{e \in E(T)} w_{\{p,q\}}(e) \cdot \ell(e)$$
for every two points $p, q \in |D|$.
We show that $d(p, q) \in \RR_{\geq 0}$ for every $p, q \in |D|$:
Note that at most four elements of $E_n$ can have elements of $\{\vstart(p), \vstart(q), \vend(p), \vend(q)\}$ on both sides for every $n \in \NN$.
Thus all but at most four elements of $E_n$ have weight zero.
This implies $d(p,q) = \sum_{n \in \NN} \sum_{e \in E_n}  w_{\{p,q\}}(e) \cdot \frac{1}{2^n} \leq \sum_{n \in \NN} \frac{4}{2^n} = 4$.
Since $w_{\{p,q\}}(e) \cdot \ell(e)$ is non-negative for every $e \in E(T)$, $d$ maps to $\RR_{\geq 0}$.

We show that $d$ is indeed a metric on $|D|$.
The map $d$ is symmetric since the weight of an edge is symmetric.
\begin{lemma}
	The map $d$ is positive-definite.
\end{lemma}
\begin{proof}
	For $p = q$ only cases \labelcref{itm:weight_1} and \labelcref{itm:weight_4} apply.
	In both cases $w_{\{p,q\}}(e) = 0$ for every $e \in E(T)$.
	Thus $d(p,q) = 0$ holds.
	
	For $p \neq q$ we show that $d(p, q) > 0$ holds.
	More precisely, we prove that there exists an edge with non-zero weight.
	Since $p \neq q$ there exist two distinct elements $x, y \in \{\vstart(p), \vstart(q), \vend(p), \vend(q)\}$.
	Let $e$ be an edge of the $x$--$y$~path in $T$.
	Then $x$ and $y$ witness that $e$ is not of type~\labelcref{itm:weight_1}.
	If $e$ is of type \labelcref{itm:weight_3}, $w_{\{p,q\}}(e) > 0$ holds.
	If $e$ is of type \labelcref{itm:weight_2} or \labelcref{itm:weight_5}, $w_{\{p,q\}}(e) > 0$ holds since $\lambda(p), \lambda(q) \in (0,1)$ whenever they are defined.
	Thus we can assume that $e$ is of type \labelcref{itm:weight_4} and $\lambda(p) = \lambda(q)$.
	Since $p \neq q$, either $\vstart(p) \neq \vstart(q)$ or $\vend(p) \neq \vend(q)$.
	Then one side of $T -e$ contains an edge $f \in E(T)$ of type~\labelcref{itm:weight_2}, and thus $w_{\{p,q\}}(f) > 0$.
\end{proof}

\begin{lemma}
	The map $d$ satisfies the triangle-inequality.
\end{lemma}
\begin{proof}
	Let $o, p, q$ be three arbitrary points in $|D|$.
	We show that $w_{\{p,q\}}(e) \leq w_{\{p, o\}}(e) + w_{\{o, q\}}(e)$ holds for every $e \in E(T)$.
	By definition of $d$, this implies that $d$ satisfies the triangle-inequality.
	Let $e \in E(T)$ be arbitrary.
	\begin{description}
		\item[Case 1: $e$ is of type~\labelcref{itm:weight_1} with respect to $p$ and $q$] Since $w_{\{p,q\}}(e) = 0$, the desired inequality is satisfied.
		\item[Case 2: $e$ is of type~\labelcref{itm:weight_2} with respect to $p$ and $q$] By symmetry of $d$, we can assume that $x \in \{\vstart(p), \vend(p)\}$.
			If $\vstart(o)$ and $\vend(o)$ are contained in the same side of $T -e$ as $x$, $e$ is of type \labelcref{itm:weight_3} with respect to $o$ and $q$.
			Thus the desired inequality is satisfied since $w_{\{o,q\}}(e) = 1$.
			If $\vstart(o)$ and $\vend(o)$ are contained in the opposite side of $T -e$ as $x$, $e$ is of type \labelcref{itm:weight_2} with respect to $p$ and $o$.
			Thus the desired inequality is satisfied since $w_{\{p,q\}}(e) = w_{\{p, o\}}(e)$.
			We can assume that $\vstart(o)$ and $\vend(o)$ are contained in distinct sides of $T - e$.
			\begin{itemize}
				\item If $\vstart(p) = x$, and $\vstart(p)$ and $\vstart(o)$ are contained in the same side of $T - e$:
					$$w_{\{p,q\}}(e) =  1 - \lambda(p) \leq |\lambda(p) - \lambda(o)| + (1 - \lambda(o)) =  w_{\{p, o\}}(e) + w_{\{o, q\}}(e).$$
				\item If $\vstart(p) = x$, and $\vstart(p)$ and $\vend(o)$ are contained in the same side of $T - e$:
					$$w_{\{p,q\}}(e) = 1 - \lambda(p) \leq (1 - |\lambda(p) - \lambda(o)|) + \lambda(o) =  w_{\{p, o\}}(e) + w_{\{o, q\}}(e).$$
				\item If $\vend(p) = x$, and $\vend(p)$ and $\vstart(o)$ are contained in the same side of $T - e$:
					$$w_{\{p,q\}}(e) = \lambda(p) \leq (1 - |\lambda(p) - \lambda(o)|) + (1 - \lambda(o)) =  w_{\{p, o\}}(e) + w_{\{o, q\}}(e).$$
				\item If $\vend(p) = x$, and $\vend(p)$ and $\vend(o)$ are contained in the same side of $T - e$:
					$$w_{\{p,q\}}(e) = \lambda(p) \leq |\lambda(p) - \lambda(o)| + \lambda(o) =  w_{\{p, o\}}(e) + w_{\{o, q\}}(e).$$
			\end{itemize}
		\item[Case 3: $e$ is of type~\labelcref{itm:weight_3} with respect to $p$ and $q$] If $\vstart(o)$ and $\vend(o)$ are contained in the same side of $T - e$, the edge $e$ is of type~\labelcref{itm:weight_3} with respect to either $p$ and $o$ or $o$ and $q$.
		Thus $1 \leq w_{\{p, o\}}(e) + w_{\{o, q\}}(e)$, which implies the desired inequality.
		
		We can assume that $\vstart(o)$ and $\vend(o)$ are contained in distinct sides of $T - e$.
		By symmetry of $d$, we can assume without loss of generality that $\vstart(o)$ and $\vstart(p)$ are contained in the same side of $T - e$.
		Thus
		$$w_{\{p,q\}}(e) = 1 \leq \lambda(o) + (1 - \lambda(o)) =  w_{\{p, o\}}(e) + w_{\{o, q\}}(e).$$
		\item[Case 4: $e$ is of type~\labelcref{itm:weight_4} with respect to $p$ and $q$]
		\begin{itemize}
			\item If $\vstart(o)$ and $\vend(o)$ are contained in the same side of $T - e$ as $\vstart(p)$ and $\vstart(q)$:
		$$w_{\{p,q\}}(e) = |\lambda(p) - \lambda(q)| \leq \lambda(p) + \lambda(q) =  w_{\{p, o\}}(e) + w_{\{o, q\}}(e).$$
			\item If $\vstart(o)$ and $\vend(o)$ are contained in the same side of $T - e$ as $\vend(p)$ and $\vend(q)$:
		$$w_{\{p,q\}}(e) = |\lambda(p) - \lambda(q)| \leq (1 - \lambda(p)) + (1 - \lambda(q)) =  w_{\{p, o\}}(e) + w_{\{o, q\}}(e).$$
			\item If $\vstart(o)$ is contained in the same side of $T - e$ as $\vstart(p)$ and $\vstart(q)$, and $\vend(o)$ is contained in the opposite side:
				$$w_{\{p,q\}}(e) = |\lambda(p) - \lambda(q)| \leq |\lambda(p) - \lambda(o)| + |\lambda(o) - \lambda(q)| =  w_{\{p, o\}}(e) + w_{\{o, q\}}(e).$$
			\item If $\vend(o)$ is contained in the same side of $T - e$ as $\vstart(p)$ and $\vstart(q)$, and $\vstart(o)$ is contained in the opposite side:
$$w_{\{p,q\}}(e) = |\lambda(p) - \lambda(q)| \leq (1 - |\lambda(p) - \lambda(o)|) + (1 - |\lambda(o) - \lambda(q)|) =  w_{\{p, o\}}(e) + w_{\{o, q\}}(e).$$
		\end{itemize}
	\item[Case 5: $e$ is of type~\labelcref{itm:weight_5} with respect to $p$ and $q$]
		\begin{itemize}
			\item If $\vstart(o)$ and $\vend(o)$ are contained in the same side of $T - e$ as $\vstart(p)$ and $\vend(q)$:
			$$w_{\{p,q\}}(e) =  1- |\lambda(p) - \lambda(q)| \leq \lambda(p) + (1 - \lambda(q)) =  w_{\{p, o\}}(e) + w_{\{o, q\}}(e).$$
			\item If $\vstart(o)$ and $\vend(o)$ are contained in the same side of $T - e$ as $\vend(p)$ and $\vstart(q)$:
			$$w_{\{p,q\}}(e) =1 - |\lambda(p) - \lambda(q)| \leq (1 - \lambda(p)) + \lambda(q) =  w_{\{p, o\}}(e) + w_{\{o, q\}}(e).$$
			\item If $\vstart(o)$ is contained in the same side of $T - e$ as $\vstart(p)$ and $\vend(q)$, and $\vend(o)$ is contained in the opposite side:
			$$w_{\{p,q\}}(e) = 1- |\lambda(p) - \lambda(q)| \leq |\lambda(p) - \lambda(o)| + (1 - |\lambda(o) - \lambda(q)|) =  w_{\{p, o\}}(e) + w_{\{o, q\}}(e).$$
			\item If $\vend(o)$ is contained in the same side of $T - e$ as $\vstart(p)$ and $\vend(q)$, and $\vstart(o)$ is contained in the opposite side:
			$$w_{\{p,q\}}(e) =1- |\lambda(p) - \lambda(q)| \leq (1 - |\lambda(p) - \lambda(o)|) + |\lambda(o) - \lambda(q)| =  w_{\{p, o\}}(e) + w_{\{o, q\}}(e).$$
		\end{itemize}
	\end{description}
	This completes the proof.	
\end{proof}
Thus $d$ is indeed a metric on $|D|$.
It remains to prove:

\begin{lemma}
	The metric $d$ induces the topology \dmtop.
\end{lemma}
We remark that for every vertex $v$ in $D$ there exists a positive lower bound for the length of edges incident with $v$.
\begin{proof}
	Let $p$ be an arbitrary point in $|D|$.
	We begin by proving that for every small basic open set $O$ in \dmtop~containing $p$ there exists $\delta > 0$ such that $B_\delta(p) \subseteq O$.
	\begin{description}
		\item[$p$ is a vertex] Then $O$ is a uniform star of radius $\epsilon > 0$ around $p$.
			Let $\ell$ be a positive lower bound for the length of edges incident with $p$.
			We set $\delta:= \epsilon \cdot \ell$ and show, given an arbitrary point $q \in B_\delta(p) \setminus \{p\}$, that $q \in O$.
			
			Either $\vstart(q)=p$ or $\vend(q) = p$:
			Otherwise there exist either an edge of $T$ incident with $p$ of type~\labelcref{itm:weight_3} or two edges incident with $p$ of type~\labelcref{itm:weight_2} with respect to $p$ and $q$.
			Note that in the latter case these edges have weight $\lambda(q)$ and $1 - \lambda(q)$.
			In both cases $d(p,q) \geq \ell$, a contradiction.
			
			Thus there exists an edge of type~\labelcref{itm:weight_2} with respect to $p$ and $q$, where $x \in \{ \vstart(q), \vend(q)\}$.
			If $x = \vstart(q)$, $1 - \lambda(q) \leq \delta < \epsilon$.
			If $x = \vend(q)$, $\lambda(q) \leq \delta < \epsilon$.
			Thus in both cases, $q$ is contained in the uniform star of radius $\epsilon$ around $p$.
			
				\item[$p$ is an end] We can assume that $O$ is a basic open set of the form $\hat{C}_\epsilon(\ODown{t}_T, p)$  for some $1 > \epsilon > 0$ and some $t \in V(T) \setminus \{r\}$ with $p \in \Omega(D[\Up{t}_T])$.
				Note that $C(\ODown{t}_T, p ) = D[\Up{t}_T]$.
				Let $\ell$ be the length of the edge $e$ of $T$ incident to $t$ and its parent.
				We set $\delta := \epsilon \cdot \ell$ and show, given an arbitrary point $q \in B_\delta(p) \setminus \{p\}$, that $q \in O$.
				
				Since $\vstart(p) = \vend(p)$, $e$ is not of type~\labelcref{itm:weight_4} or \labelcref{itm:weight_5}.
				Since $d(p,q) < \delta$, $w_{\{p,q\}}(e) < \epsilon$ and therefore $e$ is not of type~\labelcref{itm:weight_3}.
				This implies that at least one of $\vstart(q)$ and $\vend(q)$ is contained in $\hat{C}_\epsilon(\ODown{t}_T, p)$.
				If one of $\vstart(q)$ and $\vend(q)$ is not contained in $\hat{C}_\epsilon(\ODown{t}_T, p)$, then $e$ is of type~\labelcref{itm:weight_2}.
				Thus $w_{\{p,q\}}(e) < \epsilon$ implies that $(1 - \lambda(q)) < \epsilon$ for $\vstart(q) \notin \hat{C}_\epsilon(\ODown{t}_T, p)$ and $\lambda(q) < \epsilon$ for $\vend(q) \notin \hat{C}_\epsilon(\ODown{t}_T, p)$.
				Thus $q \in \hat{C}_\epsilon(\ODown{t}_T, p)$.
				
				\item[$p$ is an inner point of a (limit) edge]
				First, we assume that $p$ is an inner point of a limit edge.
				We can assume that $O$ is a basic open set of the form $\hat{E}_{\epsilon,p}(X, \vstart(p) \vend(p))$ for $1 > \epsilon > 0$ and $X \subseteq V(D)$ such that for every end $\omega$ in $\{\vstart(p), \vend(p)\}$ there exists $t_\omega \in V(T) \setminus\{r\}$ with $C(X, \omega) = C(\ODown{t_\omega}_T, \omega) = D([\Up{t_\omega}_T])$.
				
				Let $\rho:= \min(\lambda(p), 1 - \lambda(p))$.
				For a vertex $y$ in $\{\vstart(p), \vend(p)\}$, let $\ell_y$ be a positive lower bound for the length of edges incident with $y$.
				For an end $y$ in $\{\vstart(p), \vend(p)\}$, let $\ell_y$ be the length of the edge $e_y$ of $T$ incident with $t_y$ and its parent.
				We set $\delta:= \frac{\rho}{2} \cdot \epsilon \cdot \min(\ell_{\vstart(p)}, \ell_{\vend(p)})$ and show, given an arbitrary point $q \in B_\delta(p) \setminus \{p\}$, that $q \in O$.
				
				We prove that $\vstart(p) = \vstart(q)$ and $|\lambda(p) - \lambda(q)| < \epsilon$ if $\vstart(p)$ is a vertex:
				The edge $e$ of the $\vstart(p)$--$\vend(p)$~path or ray in $T$ incident with $\vstart(p)$ is not of type~\labelcref{itm:weight_1,itm:weight_2,itm:weight_3} or \labelcref{itm:weight_5} since $w_{\{p,q\}}(e) < \rho$ and $\vstart(p), \vend(p)$ are on distinct sides of $T - e$.
				Thus $e$ is of type~\labelcref{itm:weight_4} and $|\lambda(p) - \lambda(q)| \leq \frac{\rho}{2}, \epsilon$.
				This implies $\lambda(q), (1 - \lambda(q)) \geq \frac{\rho}{2}$.
				Thus no edge $f$ incident with $\vstart(p)$ is of type~\labelcref{itm:weight_2}, since $w_{\{p,q\}}(f) < \frac{\rho}{2}$.
				We can deduce $\vstart(p) = \vstart(q)$.
				
				 We show that $\vstart(q) \in V(C(X, \vstart(p))) \cup \Omega(C(X, \vstart(p)))$ and $|\lambda(p) - \lambda(q)| < \epsilon$ if $\vstart(p)$ is an end:
				 The edge $e_{\vstart(p)}$ is not of type~\labelcref{itm:weight_1,itm:weight_2,itm:weight_3,itm:weight_5} since $w_{\{p,q\}}(e_{\vstart(p)})< \rho \leq 1 - |\lambda(p) -\lambda(q)| \leq 1$ and $\vstart(p), \vend(p)$ are on distinct sides of $T - e_{\vstart(p)}$.
				 Thus $e_{\vstart(p)}$ is of type~\labelcref{itm:weight_4} and $|\lambda(p) - \lambda(q)| < \epsilon$.
				 Furthermore, $\vstart(q) \in V(C(X, \vstart(p))) \cup \Omega(C(X, \vstart(p)))$.
				
				Similarly, we can prove that $\vend(p) = \vend(q)$ if $\vend(p)$ is a vertex and $\vend(q) \in V(C(X, \vend(p))) \cup \Omega(C(X, \vstart(p)))$ if $\vend(p)$ is an end.
				Thus $q \in \hat{E}_{\epsilon,p}(X, \vstart(p) \vend(p))$.
				
				Second, we assume that $p$ is an inner point of an edge and assume that $O$ is a basic open set of $p$ for some $\epsilon > 0$.
				We define $\delta$ as before and show, given an arbitrary point $q \in B_\delta(p) \setminus \{p\}$, that $q \in O$.
				We remark that, by the same argument as before, $\vstart(p) = \vstart(q)$ and $\vend(p)=\vstart(q)$.
				Then there exists an edge $e$ of $T$ incident with $p$ of type~\labelcref{itm:weight_4}.
				Thus $|\lambda(p) - \lambda(q)| \cdot \ell(e) < \delta \leq \epsilon \cdot \min(\ell_{\vstart(p)}, \ell_{\vend(p)})$.
				Since $\ell(e) \geq \ell_{\vstart(p)}$, we obtain $|\lambda(p) - \lambda(q)| < \epsilon$.
				This shows that $q \in O$.
				
	\end{description}
	Now we show that for every $1 > \delta > 0$ there exists a basic open set $O$ in \dmtop~containing $p$ such that $O \subseteq B_\delta(p)$.
Then the metric $d$ induces the topology \dmtop.
	\begin{description}
		\item[$p$ is a vertex]
		We set $\epsilon := \frac{\delta}{2}$.
		Let $o$ be some point in the uniform star of radius $\epsilon$ around $p$.
		Then every edge of $T$ with non-zero weight with respect to $p$ and $o$ is of type~\labelcref{itm:weight_2}.
		Furthermore, $E_n$ contains at most two edge of non-zero weight with respect to $p$ and $o$ for every $n \in \NN$.
		Thus $d(p, o) < \sum_{n \in \NN} 2 \cdot \frac{1}{2^n} \cdot \epsilon = \delta$.
		
		\item[$p$ is an end] 
		Let $n \in \NN$ with $\frac{1}{2^n} < \delta$ and let $t$ be the unique vertex of $T$ of distance $n+3$ in $T$ to the root with $p \in \Omega(D[\Up{t}_T])$.
		Given some $q \in \hat{C}_{\frac{1}{2^{n+2}}}(\ODown{t}_T, p)$, we show that $q \in B_{\frac{1}{2^n}}(p)$.
		
		By definition of $\hat{C}_{\frac{1}{2^{n+2}}}(\ODown{t}_T, p)$,
		either $\vstart(q)$ or $\vend(q)$ is contained in $V(C(\ODown{t}_T, p)) \cup \Omega(C(\ODown{t}_T, p))$.
		Note that $C(\ODown{t}_T, p) = D[\Up{t}_T]$.
		Thus each $E_m$ for $m \leq n +3$ contains at most two edges of non-zero weight with respect to $p$ and $q$ and these edges are of type~\labelcref{itm:weight_2}.
		Further, there are at most four edges of non-zero weight in $E_m$ for $m > n +3$.
		
		If $\vstart(q)$ and $\vend(q)$ are contained in $V(C(\ODown{t}_T, p)) \cup \Omega(C(\ODown{t}_T, p))$, then $d(p,q) \leq \sum_{m \leq n +3 } 0 + \sum_{m > n+3} 4 \cdot 1 \cdot \frac{1}{2^m} < \frac{1}{2^n}$.
		Thus we can assume that exactly one of $\vstart(q)$ and $\vend(q)$ is contained in $V(C(\ODown{t}_T, p)) \cup \Omega(C(\ODown{t}_T, p))$.
		If $\vstart(q) \in V(C(\ODown{t}_T, p)) \cup \Omega(C(\ODown{t}_T, p))$, then $\lambda(q) < \frac{1}{2^{n+2}}$ and thus $d(p,q) \leq \sum_{m \leq n +3 } 2 \cdot \frac{1}{2^{n+2}} \cdot \frac{1}{2^m} + \sum_{m > n+3} 4 \cdot 1 \cdot \frac{1}{2^m} < \frac{1}{2^n}$.
		If $\vend(q) \in V(C(\ODown{t}_T, p)) \cup \Omega(C(\ODown{t}_T, p))$, then $1 - \lambda(q) < \frac{1}{2^{n+2}}$ and thus similarly $d(p,q) < \frac{1}{2^n}$.
		
		\item[$p$ is an inner point of an edge]
		Let $q$ be some point in the $\frac{\delta}{2}$-ball around $p$ in $\vstart(p) \vend(p)$.
		We show $q \in B_\delta(p)$.
		
		Note that every edge of $T$ is either of type~\labelcref{itm:weight_1} or of type~\labelcref{itm:weight_4}.
		Further there are at most two edges in $E_m$ of type~\labelcref{itm:weight_4} for every $m \in \NN$.
		Thus $d(p,q) \leq \sum_{m \in \NN} 2 \cdot |\lambda(p) -\lambda(q)| \cdot \frac{1}{2^m}  < \delta$.
		
		\item[$p$ is an inner point of a limit edge]
		Let $n \in \NN$ with $\frac{1}{2^n} < \delta$ and let $X \subseteq V(D)$ such that $C(X, \omega)$ contains only vertices of distance at least $n+3$ to $r$ in $T$ for every end $\omega$ in $\{\vstart(p), \vend(p)\}$ and such that every vertex in $\{\vstart(p), \vend(p)\}$ is in $X$.
		Further, let $q$ be some point in $\hat{E}_{\frac{1}{2^{n+2}}, p}(X, \vstart(p) \vend(p))$.
		We show $q \in B_\delta(p)$.
		
		For every $m \leq n+3$ the edges in $E_m$ are either of type~\labelcref{itm:weight_1} or of type~\labelcref{itm:weight_4} by choice of $X$.
		Further, there are at most two edges in $E_m$ of type~\labelcref{itm:weight_4}.
		Further there are at most four edges in $E_m$ that are of type~\labelcref{itm:weight_2,itm:weight_3,itm:weight_4,itm:weight_5}, i.e.\ with non-zero weight, for every $m > n+3$.
		Thus $d(p,q) \leq \sum_{m \leq n+3} 2 \cdot |\lambda(p) -\lambda(q)| \cdot \frac{1}{2^m} + \sum_{m > n +3} 4 \cdot \frac{1}{2^m} < \delta$. \qedhere
	\end{description}
\end{proof}
Thus $d$ is indeed the desired metric inducing \dmtop, which proves the backward implication of \cref{thm:metrizable}.

\section*{Acknowledgement}

The author gratefully acknowledges support by a doctoral scholarship of the Studienstiftung des deutschen Volkes.

\bibliography{ref.bib}

\end{document}